\documentclass[pdflatex,sn-mathphys-num]{sn-jnl}

\usepackage{graphicx}%
\usepackage{multirow}%
\usepackage{amsmath,amssymb,amsfonts}%
\usepackage{amsthm}%
\usepackage{mathrsfs}%
\usepackage[title]{appendix}%
\usepackage{xcolor}%
\usepackage{textcomp}%
\usepackage{manyfoot}%
\usepackage{booktabs}%
\usepackage{algorithm}%
\usepackage{algorithmicx}%
\usepackage{algpseudocode}%
\usepackage{listings}%

\theoremstyle{thmstyleone}%
\newtheorem{theorem}{Theorem}
\newtheorem{proposition}[theorem]{Proposition}%

\newtheorem{lemma}[theorem]{Lemma}%

\newtheorem{corollary}[theorem]{Corollary}%

\theoremstyle{thmstyletwo}%
\newtheorem{remark}{Remark}%

\theoremstyle{thmstylethree}%

\begin{document}

\title[Further Results on the Quadratic Embedding Constants of Corona Graphs]{Further Results on the Quadratic Embedding Constants of Corona Graphs}


\author[1]{\fnm{Ferdi} }\email{ferdimath123@gmail.com}

\author*[1,2]{\fnm{Edy Tri} \sur{Baskoro}}\email{ebaskoro@itb.ac.id}

\author[1]{\fnm{Aditya Purwa} \sur{Santika}}\email{aditps@itb.ac.id}

\affil[1]{\orgdiv{Combinatorial Mathematics Research Group}, \orgname{Faculty of Mathematics and Natural Sciences, Institut Teknologi Bandung}, \country{Indonesia}}

\affil[2]{\orgdiv{Center for Research Collaboration on Graph Theory and Combinatorics}, \country{Indonesia}}


\abstract{The quadratic embedding constant (QEC) is a numerical invariant 
associated with quadratic embeddings of graphs into Hilbert spaces, 
and it is characterized in terms of the distance matrix. 
For corona graphs $G\odot H$, a general expression for 
$\mathrm{QEC}(G\odot H)$ can be described using $\mathrm{QEC}(G)$ 
together with spectral properties of $H$. 
However, this expression involves an additional spectral contribution 
determined by the adjacency matrix of $H$.  In this paper, we analyze this contribution and provide an explicit 
description of the associated set $\Gamma$, allowing us to determine 
the quantity $\gamma = \max \Gamma$ that appears in the general formula 
for $\mathrm{QEC}(G\odot H)$. As applications, we compute the quadratic 
embedding constants for corona graphs of the form $G\odot H$ where $H$ is a regular graph. 
Finally, we provide conditions on $G$ and $H$ under which the quadratic
embedding constant of $G\odot H$ coincides with the second largest
eigenvalue of the distance matrix.}




\maketitle


\section{Introduction}

Let $G=(V,E)$ be a simple connected undirected graph with $|V|=n\ge 2$. 
The distance matrix of $G$ is defined as $D_G=[d(x,y)]_{x,y\in V}$, where 
$d(x,y)$ denotes the distance between vertices $x$ and $y$ in $G$. A map 
$\varphi:V\to\mathcal{H}$ into a Hilbert space $\mathcal{H}$ is called a 
\emph{quadratic embedding} of $G$ if
\[
\|\varphi(x)-\varphi(y)\|^2 = d(x,y), \quad x,y\in V.
\]
A graph that admits such an embedding is said to belong to the \emph{QE class}. 
By Schoenberg's theorem \cite{Schoenberg1935,Schoenberg1937}, a connected graph 
$G$ belongs to the QE class if and only if its distance matrix $D$ is conditionally 
negative definite.

Obata and Zakiyyah \cite{Obata-Zakiyyah2018} introduced the
\emph{quadratic embedding constant} of a graph $G$, defined by
\[
\mathrm{QEC}(G)
=
\max
\bigl\{
\langle f,D_G f\rangle
\; ; \;
f\in C(V),\,
\langle f,f\rangle=1,\,
\langle \mathbf{1},f\rangle=0
\bigr\}.
\]
The graph $G$ belongs to the QE class if and only if $\mathrm{QEC}(G)\le0$,
and thus $\mathrm{QEC}(G)$ characterizes the existence of a quadratic
embedding of $G$. Graphs in the QE class play an important role in quantum probability theory 
and non-commutative harmonic analysis \cite{Haagerup1979, Hora-Obata2007, Obata2007, Schoenberg1938}.

The quadratic embedding constant has been computed for various graph families, 
including complete graphs and cycles \cite{Obata-Zakiyyah2018}, paths 
\cite{Mlotkowski2022}, wheels \cite{Obata2017}, fan graphs 
\cite{Mlotkowski-Obata2025a,Mlotkowski-Obata2025b}, complete multipartite graph \cite{Obata2023}, cluster graphs 
\cite{Choudhury-Nandi2023}, so on. A natural direction of research is to study how 
$\mathrm{QEC}(G)$ behaves under graph operations. Investigations along these lines 
have been carried out for Cartesian products \cite{Obata-Zakiyyah2018,Choudhury-Nandi2025}, star products 
\cite{Baskoro-Obata2021,Mlotkowski-Obata2020}, joins 
\cite{Lou-Obata-Huang2022,Mlotkowski-Obata2025a,Choudhury-Nandi2025}, and lexicographic products 
\cite{Lou-Obata-Huang2022}, establishing a relation between the graph structure and the spectral properties of the distance matrix.

One such construction is the corona graph. 
Given graphs $G$ and $H$, the corona graph $G\odot H$ is obtained by 
taking one copy of $G$ and $|V(G)|$ copies of $H$, and joining the 
$i$-th vertex of $G$ to every vertex in the $i$-th copy of $H$ 
(see \cite{Frucht-Harary1970}). The quadratic embedding constant 
of corona graphs has been studied in \cite{Choudhury-Nandi2023,Ferdi-Baskoro-Obata-Santika2025}, where 
a general expression for $\mathrm{QEC}(G\odot H)$ was obtained in terms of 
$\mathrm{QEC}(G)$ and the spectral properties of $H$. Specifically, this value can 
be described using an analytic function $\psi_H$ associated with the adjacency matrix of $H$. However, the general formula involves an additional spectral contribution arising 
from the set
\[
\Gamma
=
\lambda(\mathcal{S})\cap\left\{
\lambda \in \mathbb{R} \; ; \; \det(A_H+\lambda+2)=0
\right\},
\]
whose role has not been fully analyzed. Previous results often imposed auxiliary 
assumptions to exclude this contribution, leaving the general formula incomplete. 
The main goal of the present paper is to clarify this spectral term and provide explicit 
expressions for the quadratic embedding constant of corona graphs in concrete cases. 
We give a precise description of the set $\Gamma$ and determine
\[
\gamma = \max \Gamma,
\]
which arises in the general formula for $\mathrm{QEC}(G\odot H)$ and completes the analysis of the spectral term. As an application, we consider corona graphs $G\odot H$ with $H$ regular. For this class of graphs, the set $\Gamma$ admits an explicit description, which leads to explicit formulas for $\mathrm{QEC}(G\odot H)$

An important observation is that the quadratic embedding constant is
closely related to the eigenvalues of the distance matrix. Since
$\mathrm{QEC}(G)$ is defined via the distance matrix, it is natural to
examine this relation in terms of the spectrum of $D_G$. Let
\[
\delta_1(D_G) \ge \delta_2(D_G) \ge \dots \ge \delta_n(D_G)
\]
be the eigenvalues of $D_G$. Then the min--max principle for eigenvalues of symmetry matrices \cite{Horn-Johnson2013}, gives
\[
\delta_2(D_G) \le \mathrm{QEC}(G) < \delta_1(D_G),
\]
so a natural question is when the quadratic embedding constant coincides with the 
second largest eigenvalue, that is, when
\[
\mathrm{QEC}(G) = \delta_2(D_G).
\]
This equality is known to hold for transmission-regular and distance-regular graphs 
\cite{Obata-Zakiyyah2018}, as well as for even paths \cite{Mlotkowski2022}, double star graphs 
\cite{Choudhury-Nandi2023}, and strongly-regular graph \cite{Obata2025}. However, a general characterization for arbitrary graphs 
remains open. In particular, it is not yet fully understood under which conditions
\[
\mathrm{QEC}(G\odot H) = \delta_2(D[G\odot H])
\]
holds. In this paper, we provide an answer to this question for corona graphs.

The paper is organized as follows. In Section~2 we review basic definitions and known 
results on quadratic embedding constants of corona graphs. Section~3 provides an 
explicit description of the set $\Gamma$. In Section~4, we compute the quadratic 
embedding constant for $G\odot H$ with $H$ being regular graph. Finally, in Section~5 provides conditions on $G$ and $H$ under which the quadratic
embedding constant of $G\odot H$ coincides with the second largest
eigenvalue of the distance matrix.


\section{Quadratic Embedding Constants of Corona Graphs in General}

In this section, we recall the definition of corona graphs \cite{Frucht-Harary1970} 
together with the general formula for the quadratic embedding constant 
established by \cite{Ferdi-Baskoro-Obata-Santika2025}.

Let $G=(V_1,E_1)$ and $H=(V_2,E_2)$ be graphs with $V_1\cap V_2=\emptyset$.
The corona graph $G\odot H$ is obtained by joining each vertex of $G$
to all vertices in a corresponding copy of $H$.
The graph $G\odot H$ is connected whenever $G$ is connected. By definition, the vertex set of $G\odot H$ consists of $V_1$
together with $|V_1|$ copies of $V_2$.
Let $o \notin V_2$. Then the vertex set can be written as
\begin{equation}\label{eq:vertex-corona}
V(G\odot H)
=
V_1 \times (\{o\} \cup V_2)
=
\{(x,i) \,;\, x \in V_1,\ i \in \{o\} \cup V_2\}.
\end{equation}
Let $d_G$ and $d_{\Tilde{H}}$ denote the graph distances on
$G$ and $\Tilde{H}=K_1+H$, respectively.
The graph distance on $G\odot H$ is given by
\[
d_{G\odot H}((x,i),(y,j))=
\begin{cases}
d_{\Tilde{H}}(i,j), & \text{if } x=y, \\
d_{\Tilde{H}}(i,o)+d_G(x,y)+d_{\Tilde{H}}(o,j), & \text{if } x\neq y.
\end{cases}
\]
This representation leads to the following expression for the distance matrix:
\begin{equation}\label{eq:distance-corona}
D_{G\odot H}
=
D_G\otimes 
\begin{bmatrix} 
J & J \\ 
J & J 
\end{bmatrix}
+
I_G \otimes 
\begin{bmatrix} 
0 & 0 \\ 
0 & -2I-A_H 
\end{bmatrix}
+
J_G \otimes 
\begin{bmatrix} 
0 & J \\ 
J & 2J 
\end{bmatrix}.
\end{equation}
For simplicity the vertex set of $G\odot H$ is denoted by $V$
and the distance matrix by $D$. By general theory (see \cite{Obata-Zakiyyah2018}), we have
\begin{equation}\label{03eqn:starting formula for QEC(1)}
\mathrm{QEC}(G\odot H)=\max \lambda(\mathcal{S}),
\end{equation}
where $\mathcal{S}=\mathcal{S}_{G\odot H}$ be the set of solutions 
$(f,\lambda,\mu)\in C(V)\times \mathbb{R}\times \mathbb{R}$
to equations:
\begin{align}
(D-\lambda)f &=\frac{\mu}{2}\mathbf{1}, 
\label{02eqn:basic equation (1)} \\
\langle \mathbf{1},f\rangle &=0, 
\label{02eqn:basic equation (2)} \\
\langle f,f\rangle &=1.
\label{02eqn:basic equation (3)}
\end{align}
We will solve \eqref{02eqn:basic equation (1)}-\eqref{02eqn:basic equation (3)}. Let $\{e_i; i\in V_2\}$ be the canonical basis of $C(V_2)$. For any $f\in C(V) \cong C(V_1)\otimes (\mathbb{R}\oplus C(V_2))$
admits a unique expression of the form:
\begin{equation}\label{02eqn:expansion of f in C(V)}
f=\xi_o\otimes \begin{bmatrix} 1 \\ 0 \end{bmatrix}
+ \sum_{i\in V_2} \xi_i\otimes \begin{bmatrix} 0 \\ e_i \end{bmatrix},
\qquad
\xi_i \in C(V_1),
\quad i\in \{o\}\cup V_2\,.
\end{equation}
Then, 
\begin{align}
\langle \mathbf{1},f\rangle=0 \iff \langle \mathbf{1}, \xi_o\rangle
+\sum_{i\in V_2} \langle \mathbf{1}, \xi_i\rangle=0
\label{02eqn:basic equation (13)},\\
\langle f,f\rangle=0 \iff \langle \xi_o, \xi_o\rangle
+\sum_{i\in V_2} \langle \xi_i, \xi_i\rangle=1.
\label{02eqn:basic equation (14)}
\end{align}
To write down the equation \eqref{02eqn:basic equation (1)}, we start with
\begin{align*}
D\bigg(\xi_o\otimes \begin{bmatrix} 1 \\ 0 \end{bmatrix}\bigg)
&=D_G\xi_o\otimes \begin{bmatrix} 1 \\ \mathbf{1} \end{bmatrix}
 +J_G\xi_o\otimes \begin{bmatrix} 0 \\ \mathbf{1} \end{bmatrix},
\\
D\bigg(\xi_i\otimes \begin{bmatrix} 0 \\ e_i \end{bmatrix}\bigg)
&=D_G\xi_i\otimes \begin{bmatrix} 1 \\ \mathbf{1} \end{bmatrix}
 +\xi_i\otimes \begin{bmatrix} 0 \\ -(A_H+2)e_i \end{bmatrix}
 +J_G\xi_i\otimes \begin{bmatrix} 1 \\ 2\mathbf{1} \end{bmatrix},
\quad i\in V_2.
\end{align*}
Then we have
\begin{align*}
(D-\lambda)f&=D_G\xi_o\otimes \begin{bmatrix} 1 \\ \mathbf{1} \end{bmatrix}
 +J_G\xi_o\otimes \begin{bmatrix} 0 \\ \mathbf{1} \end{bmatrix}
 +\sum_{i\in V_2} 
  D_G\xi_i\otimes \begin{bmatrix} 1 \\ \mathbf{1} \end{bmatrix}
 +\sum_{i\in V_2} 
  \xi_i\otimes \begin{bmatrix} 0 \\ -(A_H+2)e_i \end{bmatrix} 
\\
&\qquad \qquad
 +\sum_{i\in V_2} 
  J_G\xi_i\otimes \begin{bmatrix} 1 \\ 2\mathbf{1} \end{bmatrix}
-\lambda \xi_o\otimes \begin{bmatrix} 1 \\ 0 \end{bmatrix}
-\sum_{i\in V_2} 
  \lambda \xi_i\otimes \begin{bmatrix} 0 \\ e_i \end{bmatrix}.
\end{align*}
Comparing the coefficients with respect to canonical basis
\[\begin{bmatrix}
1\\0
\end{bmatrix}, \qquad \begin{bmatrix}
0\\ e_i
\end{bmatrix}, \quad i\in V_2,\]
on both sides of the equation $(D-\lambda)f=\dfrac{\mu}{2}\mathbf{1}$, we obtain
\begin{align}
&(D_G-\lambda)\xi_o 
+\sum_{i\in V_2}(D_G\xi_i+J_G\xi_i)=\frac{\mu}{2}\,\mathbf{1},
\label{03eqn:main equation (01)}\\
&(D_G+J_G)\xi_o 
+\sum_{j\in V_2}D_G\xi_j
+\sum_{j\in V_2}2J_G\xi_j
\nonumber \\
&\qquad\qquad +\sum_{j\in V_2} \langle e_i, -(A_H+2)e_j\rangle \xi_j
-\lambda\xi_i
=\frac{\mu}{2}\,\mathbf{1},
\qquad i\in V_2.
\label{03eqn:main equation (02)} 
\end{align}
By subtracting \eqref{03eqn:main equation (01)} from \eqref{03eqn:main equation (02)} and using \eqref{02eqn:basic equation (13)}, we obtain
\begin{align}
&(D_G-J_G-\lambda)\xi_o 
+\sum_{i\in V_2}D_G\xi_i=\frac{\mu}{2}\,\mathbf{1},
\label{02eqn:basic equation (11)}
\\
&\lambda\xi_o
=\sum_{j\in V_2} \langle e_i, (A_H+2+\lambda)e_j\rangle \xi_j\,,
\qquad i\in V_2.
\label{02eqn:basic equation (12)}
\end{align}

Furthermore, the quadratic embedding constant of the corona graph $G\odot H$ was determined in \cite{Ferdi-Baskoro-Obata-Santika2025}. For convenience, we recall the result below.

\begin{theorem}[\cite{Ferdi-Baskoro-Obata-Santika2025}]\label{thm:corona-general}
Let $G=(V_1,E_1)$ be a connected graph with $|V_1|\ge 2$, 
and let $H=(V_2,E_2)$ be a graph with $|V_2|\ge 1$. 
Let $A_H$ denote the adjacency matrix of $H$ and define
\[
\psi_H(\lambda)
=
\frac{\lambda}{
1+\lambda \langle \mathbf{1}, (A_H+2+\lambda)^{-1}\mathbf{1}\rangle}.
\]
Assume that
\begin{equation}\label{eq:condition-spectrum}
-2-\psi_{H}^{-1}(\mathrm{QEC}(G))
\notin \mathrm{ev}(A_H).
\end{equation}
Then the quadratic embedding constant of the corona graph $G\odot H$ is given by
\begin{equation}\label{eq:qec-corona}
\mathrm{QEC}(G\odot H)
=
\max\{
\psi_{H}^{-1}(\mathrm{QEC}(G)),
\gamma
\},
\end{equation}
where
\begin{equation}\label{eq:gamma3}
\gamma
=\max\Gamma=
\max \lambda(\mathcal{S})\cap\left\{
\lambda \in \mathbb{R};
\det(A_H+\lambda+2)=0
\right\}.
\end{equation}
If the set $\Gamma$ is empty, we understand that $\gamma=-\infty$.
\end{theorem}

\begin{remark}
In this paper we restrict our attention to corona graphs $G \odot H$ where $G$ is a connected graph with $|V(G)| \ge 2$. 
The case $G=K_1$ is excluded since $\mathrm{QEC}(K_1)$ is not defined. 
Observe that $K_1 \odot H$ is isomorphic to the join graph $K_1+H$. 
Although this graph has a simple construction, a general expression for 
$\mathrm{QEC}(K_1+H)$ is not known in general. Some partial results are available, for example when $H$ is regular 
\cite{Lou-Obata-Huang2022} and when $H=P_n$ is a path graph 
\cite{Mlotkowski-Obata2025a,Mlotkowski-Obata2025b}.
\end{remark}

\section{Explicit Description of the Corona Graph $G\odot H$}

In \cite{Ferdi-Baskoro-Obata-Santika2025}, the case $\Gamma=\emptyset$ was completely resolved, where the quadratic embedding constant of the corona graph satisfies $\mathrm{QEC}(G\odot H)=\psi_H^{-1}(\mathrm{QEC}(G))$. Thus, in this situation the value of $\mathrm{QEC}(G\odot H)$ is entirely determined by the term $\psi_H^{-1}(\mathrm{QEC}(G))$. The present section treats the remaining case $\Gamma\ne\emptyset$. In this case, the determination of $\mathrm{QEC}(G\odot H)$ involves an additional spectral contribution arising from the set $\Gamma$ appearing in the general formula. We therefore give an explicit description of $\Gamma$ and analyze its role in the evaluation of $\mathrm{QEC}(G\odot H)$.

We first consider the values $\lambda\in \lambda(\mathcal{S})$ if and only if it appears in the solutions to the equations \eqref{02eqn:basic equation (11)}, \eqref{02eqn:basic equation (12)}, \eqref{02eqn:basic equation (13)}, and \eqref{02eqn:basic equation (14)}. 
We will determine such $\lambda$'s according to the partition
\[
\lambda(\mathcal{S})=\Gamma_1\cup \Gamma_2\cup \Gamma_3,
\]
where
\begin{align}
&\Gamma_1=-2-\lambda\in \lambda(\mathcal{S})\setminus (\mathrm{ev}(A_H)\cup \{-2\}),\\
&\Gamma_2=-2-\lambda\in \{-2\},\\
&\Gamma_3=-2-\lambda\in \mathrm{ev}(A_H)\setminus \{-2\}.
\end{align}
For convenience, we also denote
\[
\Gamma=\Gamma_2\cup \Gamma_3 .
\]

\begin{lemma}\label{lemma:Gamma1}
Let $H=(V_2,E_2)$ be a graph with adjacency matrix $A_H$, and let $\lambda\in \Gamma_1$. 
Then $\lambda\in \lambda(\mathcal{S})$ if and only if
\begin{align}\label{eq:Gamma1}
\lambda_{\max}=\psi_H^{-1}(\mathrm{QEC}(G)),
\end{align}
where
\[
\psi_H(\lambda)
=
\frac{\lambda}{
1+\lambda \langle \mathbf{1}, (A_H+2+\lambda)^{-1}\mathbf{1}\rangle}.
\]
Here $\psi_H^{-1}(\mathrm{QEC}(G))$ denotes the largest solution of the equation
$\psi_H(\lambda)=\mathrm{QEC}(G)$.
\end{lemma}

\begin{proof}
The statement follows directly from Theorem~\ref{thm:corona-general}.
\end{proof}

\begin{lemma}\label{lemma 2: Gamma_2}
Let $H=(V_2,E_2)$ be a graph with adjacency matrix $A_H$ and let 
$\lambda\in \Gamma_2$. Then $\lambda\in \lambda(\mathcal S)$
if and only if $\mathrm{QEC}(G)=0$ or $-2\in \mathrm{ev}(A_H)$.
\end{lemma}

\begin{proof}

\noindent
($\Rightarrow$)
Suppose that $\lambda\in \lambda(\mathcal S)$.
From \eqref{02eqn:basic equation (12)} we obtain
\[
\sum_{j\in V_2}\langle e_i,(A_H+2)e_j\rangle\xi_j=0, \qquad i\in V_2.
\]
If $-2\in\mathrm{ev}(A_H)$ we are done. 
Hence assume that $-2\notin\mathrm{ev}(A_H)$.
Then $A_H+2$ is invertible and hence $\xi_i=0$ for all $i\in V_2$.
Consequently \eqref{02eqn:basic equation (11)}, 
\eqref{02eqn:basic equation (13)}, and
\eqref{02eqn:basic equation (14)} reduce to
\[
D_G\xi_o=\frac{\mu}{2}\mathbf{1},\qquad
\langle \mathbf{1},\xi_o\rangle=0,\qquad
\langle \xi_o,\xi_o\rangle=1.
\]
Hence $\mathrm{QEC}(G)=0$.

\medskip

\noindent
($\Leftarrow$)
Suppose that $-2\in\mathrm{ev}(A_H)$.
Then there exists $x'\neq0$ such that
\[
(A_H+2)x'=0 .
\]
Define
\[
\mu=0,\qquad
\xi_o=0,\qquad
\xi_i=x'_i\mathbf 1 ,\quad i\in V_2 .
\]
A direct computation shows that
\eqref{02eqn:basic equation (11)}, \eqref{02eqn:basic equation (12)}, \eqref{02eqn:basic equation (13)}, and \eqref{02eqn:basic equation (14)}
are satisfied.
Hence $\lambda\in\lambda(\mathcal S)$.

Finally, assume that $\mathrm{QEC}(G)=0$.
Then there exists $x'\in C(V_1)$ such that
\[
D_Gx'=\frac{\eta}{2}\mathbf 1,\qquad
\langle \mathbf{1},x'\rangle=0,\qquad
\langle x',x'\rangle=1 .
\]
Define
\[
\mu=\eta,\qquad
\xi_o=x',\qquad
\xi_i=0,\quad i\in V_2 .
\]
Then \eqref{02eqn:basic equation (11)}, \eqref{02eqn:basic equation (12)}, \eqref{02eqn:basic equation (13)}, and \eqref{02eqn:basic equation (14)}
hold and hence $\lambda\in\lambda(\mathcal S)$.
\end{proof}
Denote by $\gamma_2$ the maximum element of $\lambda(\mathcal S)$ arising from $\Gamma_2$. 
By Lemma~\ref{lemma 2: Gamma_2}, we have
\[
\gamma_2=
\begin{cases}
0, & \text{if } \mathrm{QEC}(G)=0 \text{ or } -2\in \mathrm{ev}(A_H),\\
-\infty, & \text{otherwise}.
\end{cases}
\]
\begin{lemma}\label{lemma 3: Gamma_3}
Let $H=(V_2,E_2)$ be a graph with adjacency matrix $A_H$, and let $\lambda \in \Gamma_3$. 
Then $\lambda \in \lambda(\mathcal S)$ if and only if there exists a nonzero vector
$x=(x_i)_{i\in V_2}$ such that
\[
(A_H + 2 + \lambda)x = 0, \qquad \langle \mathbf 1, x \rangle = 0.
\]
\end{lemma}
\begin{proof}
\noindent
($\Leftarrow$)
Suppose that there exists a nonzero vector $x'=(x'_i)_{i\in V_2}$ satisfying
\[
(A_H+2+\lambda)x' = 0, \qquad
\langle \mathbf{1},x'\rangle = 0.
\]
Define
\[
\mu=0,\qquad
\xi_o=0,\qquad
\xi_i=x'_i\mathbf{1},\quad i\in V_2,
\]
where $\mathbf{1}\in C(V_1)$ is the all-ones vector. Then
\[
D_G\sum_{j\in V_2}x'_j\mathbf{1}
=
D_G\langle \mathbf{1},x'\rangle\mathbf{1}
=
0
=
\frac{\mu}{2}\mathbf{1},
\]
so \eqref{02eqn:basic equation (11)} holds. Moreover,
\[
\sum_{j\in V_2}
\langle e_i,(A_H+2+\lambda)e_j\rangle\xi_j
=
\langle e_i,(A_H+2+\lambda)x'\rangle\mathbf{1}
=
0
=
\lambda\xi_o ,
\]
so \eqref{02eqn:basic equation (12)} holds. Furthermore,
\[
\sum_{i\in V_2}\langle \mathbf{1},\xi_i\rangle
=
\sum_{i\in V_2}x'_i\langle \mathbf{1},\mathbf{1}\rangle
=
|V_1|\langle \mathbf{1},x'\rangle
=
0,
\]
and hence \eqref{02eqn:basic equation (13)} holds. 
Since $x'\neq0$, we may normalize it so that 
$\langle x',x'\rangle=\dfrac{1}{|V_1|}$. Then
\[
\sum_{i\in V_2}\langle \xi_i,\xi_i\rangle
=
\sum_{i\in V_2}(x'_i)^2\langle \mathbf{1},\mathbf{1}\rangle
=
|V_1|\langle x',x'\rangle
=
1,
\]
so \eqref{02eqn:basic equation (14)} holds. Hence $\lambda\in\lambda(\mathcal S)$.

\medskip
\noindent
($\Rightarrow$)
Suppose that $\lambda\in\lambda(\mathcal S)$. Then there exists a solution
\[
(\xi_o,\{\xi_i\}_{i\in V_2},\lambda,\mu)
\]
of \eqref{02eqn:basic equation (11)}, \eqref{02eqn:basic equation (12)}, \eqref{02eqn:basic equation (13)}, and \eqref{02eqn:basic equation (14)}.
Define
\[
x_i=\langle\mathbf{1},\xi_i\rangle,
\qquad
x=(x_i)_{i\in V_2}.
\]
Taking the inner product of \eqref{02eqn:basic equation (12)} with $\mathbf{1}$ gives
\[
(A_H+2+\lambda)x
=
\lambda\langle\mathbf{1},\xi_o\rangle\mathbf{1}.
\]
Assume that there exists no nonzero vector $x'$ satisfying
\[
(A_H+2+\lambda)x' =0,
\qquad
\langle\mathbf{1},x'\rangle=0.
\]
Then
\[
\ker(A_H+2+\lambda)\cap \mathbf{1}^\perp=\{0\}.
\]
If $\langle\mathbf{1},\xi_o\rangle\neq0$, then the above equation implies
\[
\mathbf{1}\in\operatorname{Im}(A_H+2+\lambda).
\]
Since $A_H+2+\lambda$ is symmetric,
\[
\operatorname{Im}(A_H+2+\lambda)
=
\ker(A_H+2+\lambda)^\perp,
\]
and hence $\ker(A_H+2+\lambda)\subseteq\mathbf{1}^\perp$. 
Because $-2-\lambda\in\mathrm{ev}(A_H)$, the matrix $A_H+2+\lambda$ is singular and therefore
$\ker(A_H+2+\lambda)\neq\{0\}$, which contradicts
\[
\ker(A_H+2+\lambda)\cap\mathbf{1}^\perp=\{0\}.
\]
Thus $\langle\mathbf{1},\xi_o\rangle=0$. Consequently,
\[
(A_H+2+\lambda)x=0,
\qquad
\langle\mathbf{1},x\rangle=0,
\]
so $x\in\ker(A_H+2+\lambda)\cap\mathbf{1}^\perp$, and hence $x=0$. Therefore
\[
\langle\mathbf{1},\xi_i\rangle=0,
\qquad i\in V_2\cup\{o\}.
\]
If $\xi_o=0$, then \eqref{02eqn:basic equation (12)} implies
\[
(A_H+2+\lambda)\xi_i=0
\qquad (i\in V_2).
\]
Since $\xi_i\in\mathbf{1}^\perp$, we obtain $\xi_i=0$ for all $i\in V_2$, which contradicts \eqref{02eqn:basic equation (14)}. 
Hence $\xi_o\neq0$. Taking the inner product of \eqref{02eqn:basic equation (12)} with $\xi_o$ yields
\[
(A_H+2+\lambda)y
=
\lambda\langle\xi_o,\xi_o\rangle\mathbf{1},
\qquad
y_j=\langle\xi_j,\xi_o\rangle .
\]
Thus $\mathbf{1}\in\operatorname{Im}(A_H+2+\lambda)$, which leads to the same contradiction as above. Therefore our assumption was false, and there exists a nonzero vector $x'$ satisfying
\[
(A_H+2+\lambda)x'=0,
\qquad
\langle\mathbf{1},x'\rangle=0.
\]
\end{proof}
Denote by $\gamma_3$ the maximal element of $\lambda(\mathcal S)$ arising from $\Gamma_3$, that is,
\[
\gamma_3=\max\Gamma_3\cap \lambda(\mathcal S).
\]
By Lemma~\ref{lemma 3: Gamma_3}, this can be written as
\[
\gamma_3
=\max\{\lambda\in\mathbb{R};
\exists\,x\neq0,\,
(A_H+2+\lambda)x=0,\,
\langle\mathbf{1},x\rangle=0
\}.
\]
Here we adopt the convention that $\gamma_3=-\infty$ if 
$\Gamma_3\cap \lambda(\mathcal S)=\emptyset$.

Observe that
\[\max \Gamma=\max \Gamma_2\cup \Gamma_3=\max \{\gamma_2,\gamma_3\}.\]
Hence
\[
\max\{\psi_{H}^{-1}(\mathrm{QEC}(G)),\gamma\}
=
\max\{\psi_{H}^{-1}(\mathrm{QEC}(G)),\gamma_2,\gamma_3\}.
\]
Therefore, Theorem~\ref{thm:corona-general} can be reformulated in terms of 
$\psi_H^{-1}(\mathrm{QEC}(G)), \gamma_2$ and $\gamma_3$, as stated in the following theorem.

\begin{theorem}\label{thm:corona-general  new}
Let $G=(V_1,E_1)$ be a connected graph with $|V_1|\ge 2$, 
and let $H=(V_2,E_2)$ be a graph with $|V_2|\ge 1$. 
Let $A_H$ denote the adjacency matrix of $H$ and define
\[
\psi_H(\lambda)
=
\frac{\lambda}{
1+\lambda \langle \mathbf{1}, (A_H+2+\lambda)^{-1}\mathbf{1}\rangle}.
\]
Assume that
\begin{equation}\label{eq:condition-spectrum  new}
-2-\psi_{H}^{-1}(\mathrm{QEC}(G))
\notin \mathrm{ev}(A_H).
\end{equation}
Then we have
\begin{equation}\label{eq:qec-corona new}
\mathrm{QEC}(G\odot H)
=
\max\{
\psi_{H}^{-1}(\mathrm{QEC}(G)),
\gamma_2,\gamma_3
\}.
\end{equation}
\end{theorem}


\section{Corona graph $G\odot H$ with $H$ being regular graph}

In general, the determination of $\mathrm{QEC}(G\odot H)$ via \eqref{eq:qec-corona new} in Theorem \ref{thm:corona-general new} requires both $\psi_H^{-1}(\mathrm{QEC}(G)), \gamma_2$ and $\gamma_3$. In \cite{Ferdi-Baskoro-Obata-Santika2025}, the case where $H$ is a regular graph was investigated with emphasis on the term $\psi_H^{-1}(\mathrm{QEC}(G))$. In this section, we complete that analysis by also considering $\gamma_2$ and $\gamma_3$.

\begin{proposition}\label{prop:corona graph of G and H}
Let $G=(V_1,E_1)$ be a connected graph with $|V_1|\ge 2$, and let 
$H=(V_2,E_2)$ be a $\kappa$-regular graph on $n=|V_2|$ vertices, where $n\geq 1$ and $\kappa\geq 0$. 
Let $A_H$ be the adjacency matrix of $H$.
Assume that
\begin{equation}\label{eq:condition-spectrum graph of G and H}
-2-\psi_H^{-1}(\mathrm{QEC}(G))
\notin \mathrm{ev}(A_H).
\end{equation}
Then we have
\begin{equation}\label{eq:qec-corona graph of G and H}
\mathrm{QEC}(G\odot H)
=
\max\{
\psi_H^{-1}(\mathrm{QEC}(G)),\gamma_2,\gamma_3
\},
\end{equation}
where
\begin{equation}
\begin{split}\label{eq: psi of G and H}
\psi_H^{-1}(\mathrm{QEC}(G))&=\frac{1}{2}\bigg\{(n+1)\mathrm{QEC}(G)-(\kappa+2)\\
&\qquad+\sqrt{\left((n+1)\mathrm{QEC}(G)-(\kappa+2)\right)^2+4(\kappa+2)\mathrm{QEC}(G)}\bigg\},
\end{split}
\end{equation}
\begin{equation}\label{eq:gamma2-graph of G and H}
\gamma_2=
\begin{cases}
0, & \text{if } \mathrm{QEC}(G)=0 \text{ or } -2\in \mathrm{ev}(A_H),\\
-\infty, & \text{otherwise},
\end{cases}
\end{equation}
and
\begin{equation}\label{eq:gamma3-graph of G and H}
\gamma_3
=
-2-\min \mathrm{ev}(A_H).
\end{equation}
\end{proposition}

\begin{proof}
Note that $A_H$ is a matrix with constant row sum $\kappa$, since $A_H$ is the adjacency matrix of a regular graph of degree $\kappa$. Therefore, we obtain
\[
(A_H+2+\lambda)\mathbf{1} = (\kappa+2+\lambda)\mathbf{1},
\]
or equivalently,
\[
(A_H+2+\lambda)^{-1}\mathbf{1} = \frac{1}{\kappa+2+\lambda}\mathbf{1}.
\]
It follows that
\[
1 + \lambda \langle \mathbf{1}, (A_H+\lambda+2)^{-1} \mathbf{1} \rangle
= \frac{\lambda(n+1) + \kappa+2}{\lambda + \kappa+2}.
\]
Hence
\[\psi_{H}(\lambda)=\frac{\lambda(\lambda + \kappa+2)}{\lambda(n+1) + \kappa+2}.\]
Therefore, we obtain
\begin{align*}
\psi_{H}^{-1}(\mathrm{QEC}(G))&=\frac{1}{2}\bigg\{(n+1)\mathrm{QEC}(G)-(\kappa+2)\\& \qquad+\sqrt{\left((n+1)\mathrm{QEC}(G)-(\kappa+2)\right)^2+4(\kappa+2)\mathrm{QEC}(G)}\bigg\}.
\end{align*}
Observe that $\lambda=0$ is feasible if and only if 
$\mathrm{QEC}(G)=0$ or $-2\in \mathrm{ev}(A_H)$. Hence
\[
\gamma_2=
\begin{cases}
0, & \text{if } \mathrm{QEC}(G)=0 \text{ or } -2\in \mathrm{ev}(A_H),\\
-\infty, & \text{otherwise}.
\end{cases}
\]
Finally, assume that 
\[
\det(A_H + \lambda+2) = 0.
\] 
Consequently, there exists a nonzero vector $x$ such that
\[
A_Hx= -(\lambda+2)x.
\] 
Since \(H\) is an \(\kappa\)-regular graph, we have
\[
A_H \mathbf{1} = \kappa \mathbf{1}.
\]
Hence \(\kappa\) is a main eigenvalue of \(A_H\). Moreover, it is known
(see \cite{Rowlinson2007,Cvetkovic-Rowlinson-Simic2010}) that for a regular graph
the eigenvalue \(\kappa\) is the unique main eigenvalue and all other eigenvalues are non-main.
Therefore, for each eigenvalue \(-2-\lambda_j\), \(j=2,\dots,n\), there exists an
eigenvector \(x\) satisfying
\[
A_H x = -(\lambda_j+2) x,
\qquad
\langle \mathbf{1}, x \rangle = 0.
\]
Therefore,
\begin{align*}
\gamma_3
&= \max \lambda(\mathcal S) \cap
\bigl\{ -2-\lambda_j\; ;\;\ j\ge2 \bigr\}\\
&= -2 - \min \mathrm{ev}(A_H).
\end{align*}
Hence the assertion follows.
\end{proof}

In view of Proposition \ref{prop:corona graph of G and H}, 
$\mathrm{QEC}(G\odot H)$ is the maximum of 
$\psi_H^{-1}(\mathrm{QEC}(G)), \gamma_2$ and $\gamma_3$. 
The next three theorems characterize the cases where 
$\mathrm{QEC}(G\odot H)=\psi_H^{-1}(\mathrm{QEC}(G))$, 
$\mathrm{QEC}(G\odot H)=\gamma_2$, and 
$\mathrm{QEC}(G\odot H)=\gamma_3$.

\begin{theorem}
\label{theo II : QEC of corona graph G and H}
Let $G=(V_1,E_1)$ be a connected graph with $|V_1|\ge 2$, and let 
$H=(V_2,E_2)$ be a $\kappa$-regular graph on $n=|V_2|$ vertices, where $n\geq 1$ and $\kappa\geq 0$. 
Let $A_H$ be the adjacency matrix of $H$. Then
\begin{align*}
\mathrm{QEC}(G \odot H) 
&=\frac{1}{2}\bigg\{(n+1)\mathrm{QEC}(G)-(\kappa+2)\\
& \qquad+\sqrt{\left((n+1)\mathrm{QEC}(G)-(\kappa+2)\right)^2
+4(\kappa+2)\mathrm{QEC}(G)}\bigg\},
\end{align*}
if one of the following conditions holds:

\medskip\noindent
\textit{(i)} $-2 - \psi_{H}^{-1}(\mathrm{QEC}(G)) \notin \mathrm{ev}(A_H), 
\dfrac{\gamma_3^2 + \gamma_3(\kappa+2)}{\gamma_3(n+1) + \kappa+2} < \mathrm{QEC}(G) < 0,$
and $\gamma_3 < 0$;

\medskip\noindent
\textit{(ii)} $-2 - \psi_{H}^{-1}(\mathrm{QEC}(G)) \notin \mathrm{ev}(A_H), 
\mathrm{QEC}(G) > 0,$ and $\gamma_3 \le 0$;

\medskip\noindent
\textit{(iii)} $-2 - \psi_{H}^{-1}(\mathrm{QEC}(G)) \notin \mathrm{ev}(A_H), 
\mathrm{QEC}(G) >
\dfrac{\gamma_3^2 + \gamma_3(\kappa+2)}{\gamma_3(n+1) + \kappa+2},$
and $\gamma_3 > 0$.
\end{theorem}

\begin{proof}
Denote
\[
\gamma_3 = -2 - \min \mathrm{ev}(A_H).
\]
By Proposition~\ref{prop:corona graph of G and H}, we have
\[
\mathrm{QEC}(G \odot H)
=
\max\{\psi_{H}^{-1}(\mathrm{QEC}(G)),\gamma_2,\gamma_3\}.
\]
Hence it suffices to show that
\[
\psi_{H}^{-1}(\mathrm{QEC}(G)) \ge \gamma_2\;\; \text{ and }\;\; \psi_{H}^{-1}(\mathrm{QEC}(G)) \ge \gamma_3
\]
under each of the conditions \textit{(i)}--\textit{(iii)}. Since $\psi_{H}^{-1}(x)$ is monotonically increasing on $(-\infty,\infty)$, the comparison follows from the assumptions on $\mathrm{QEC}(G)$.

\medskip
\noindent
\textit{(i)} Suppose
\[
\frac{\gamma_3^2 + \gamma_3(\kappa+2)}
{\gamma_3(n+1)+\kappa+2}
< \mathrm{QEC}(G) < 0,
\qquad \gamma_3 < 0.
\]
By assumption, $\gamma_2=-\infty$. Moreover, since $\psi_H^{-1}$ is increasing and
\[
\psi_H^{-1}\!\left(
\frac{\gamma_3^2 + \gamma_3(\kappa+2)}
{\gamma_3(n+1)+\kappa+2}
\right)
= \gamma_3,
\]
we obtain
\[
\gamma_3
<
\psi_H^{-1}(\mathrm{QEC}(G)).
\]
Therefore,
\[
\mathrm{QEC}(G \odot H)
=
\max\{\psi_H^{-1}(\mathrm{QEC}(G)),\gamma_2,\gamma_3\}
=
\psi_H^{-1}(\mathrm{QEC}(G)).
\]

\medskip
\noindent
\textit{(ii)} Suppose
\[
\mathrm{QEC}(G) > 0,
\qquad \gamma_3 \le 0.
\]
Since $\psi_H^{-1}$ is increasing and $\psi_H^{-1}(0)=0$, we obtain
\[
\psi_H^{-1}(\mathrm{QEC}(G)) > 0 \ge \gamma_3.
\]
Moreover, we have also
\[\psi_H^{-1}(\mathrm{QEC}(G)) \ge 0 =\gamma_2.\]
Therefore,
\[
\mathrm{QEC}(G \odot H)
=
\max\{\psi_H^{-1}(\mathrm{QEC}(G)),\gamma_2,\gamma_3\}
=
\psi_H^{-1}(\mathrm{QEC}(G)).
\]

\medskip
\noindent
\textit{(iii)} Suppose
\[
\mathrm{QEC}(G) >
\frac{\gamma_3^2 + \gamma_3(\kappa+2)}
{\gamma_3(n+1)+\kappa+2},
\qquad \gamma_3 > 0.
\]
By assumption, $\gamma_2=-\infty$. Moreover, since $\psi_H^{-1}$ is increasing and
\[
\psi_H^{-1}\!\left(
\frac{\gamma_3^2 + \gamma_3(\kappa+2)}
{\gamma_3(n+1)+\kappa+2}
\right)
= \gamma_3,
\]
we obtain
\[
\gamma_3
<
\psi_H^{-1}(\mathrm{QEC}(G)).
\]
Therefore,
\[
\mathrm{QEC}(G \odot H)
=
\max\{\psi_H^{-1}(\mathrm{QEC}(G)), \gamma_2,\gamma_3\}
=
\psi_H^{-1}(\mathrm{QEC}(G)).
\]
Substituting the expression for $\psi_H^{-1}(\mathrm{QEC}(G))$ yields
\begin{align*}
\mathrm{QEC}(G \odot H) 
= \frac{1}{2} \bigg\{ 
&(n+1)\,\mathrm{QEC}(G) - (\kappa+2) \\
&+ \sqrt{\big((n+1)\,\mathrm{QEC}(G) - (\kappa+2)\big)^2 
+ 4(\kappa+2)\,\mathrm{QEC}(G)} 
\bigg\}.
\end{align*}
This completes the proof.
\end{proof}

\begin{theorem}\label{theorem corona graph with lambda equal zero}
Let $G=(V_1,E_1)$ be a connected graph with $|V_1|\ge 2$, and let 
$H=(V_2,E_2)$ be a $\kappa$-regular graph on $n=|V_2|$ vertices, where $n\ge 1$ and $\kappa\ge 0$. 
Let $A_H$ be the adjacency matrix of $H$. Then
\[
\mathrm{QEC}(G \odot H) = 0
\]
if one of the following holds:

\medskip
\noindent
\textit{(i)} $\mathrm{QEC}(G) = 0$ and $\gamma_3\le 0$;

\medskip
\noindent
\textit{(ii)} $\mathrm{QEC}(G) \le 0$ and $\gamma_3= 0$.
\end{theorem}

\begin{proof}
Denote
\[
\gamma_3 = -2 - \min \mathrm{ev}(A_H).
\]
By Proposition~\ref{prop:corona graph of G and H}, we have
\[
\mathrm{QEC}(G \odot H)
=
\max\{\psi_{H}^{-1}(\mathrm{QEC}(G)),\gamma_2,\gamma_3\}.
\]
Hence it suffices to show that
\[
\psi_{H}^{-1}(\mathrm{QEC}(G)) \le \gamma_2\;\; \text{ and }\;\; \gamma_3 \le \gamma_2
\]
under each of the conditions \textit{(i)} and \textit{(ii)}. Since $\psi_{H}^{-1}(x)$ is monotonically increasing on $(-\infty,\infty)$, the comparison follows from the assumptions on $\mathrm{QEC}(G)$.

\medskip
\noindent
\textit{(i)} Suppose
\[
\mathrm{QEC}(G) = 0, \qquad \gamma_3\le 0.
\]
Since $\psi^{-1}_H$ is increasing and $\psi^{-1}_H(0)=0$, we obtain
\[\psi_{H}^{-1}(\mathrm{QEC}(G))=0 \le \gamma_2\]
Moreover, we have also
\[\gamma_3\le 0\le \gamma_2\]
Therefore,
\[
\mathrm{QEC}(G \odot H)
=
\max\{\psi_H^{-1}(\mathrm{QEC}(G)), \gamma_2,\gamma_3\}
=
\gamma_2.
\]

\medskip
\noindent
\textit{(ii)} Suppose
\[
\mathrm{QEC}(G) \le 0, \qquad \gamma_3= 0
\]
Since $\psi^{-1}_H$ is increasing, we obtain
\[\psi_{H}^{-1}(\mathrm{QEC}(G))\le0 \le \gamma_2\]
Moreover, we have also
\[\gamma_3=0\le \gamma_2\]
Hence,
\[
\mathrm{QEC}(G \odot H)
=
\max\{\psi_H^{-1}(\mathrm{QEC}(G)), \gamma_2,\gamma_3\}
=
\gamma_2.
\]
Therefore,
\[
\mathrm{QEC}(G \odot H)=0.
\]
This completes the proof.
\end{proof}

\begin{theorem}
\label{theo III : QEC of corona graph G and H - min eigenvalue case}
Let $G=(V_1,E_1)$ be a connected graph with $|V_1|\ge 2$, and let 
$H=(V_2,E_2)$ be a $\kappa$-regular graph on $n=|V_2|$ vertices, where $n\ge 1$ and $\kappa\ge 0$. Let $A_H$ be the adjacency matrix of $H$. Then
\[
\mathrm{QEC}(G \odot H) = -2 - \min \mathrm{ev}(A_H),
\]
if one of the following conditions holds:

\medskip\noindent
\textit{(i)} 
$\mathrm{QEC}(G) \le
\dfrac{\gamma_3^2 + \gamma_3(\kappa+2)}
{\gamma_3(n+1) + \kappa+2}
\quad \text{and} \quad
\gamma_3 < 0;$

\medskip\noindent
\textit{(ii)} 
$\mathrm{QEC}(G) \le 0
\quad \text{and} \quad
\gamma_3 > 0;$

\medskip\noindent
\textit{(iii)}
$0 < \mathrm{QEC}(G) \le
\dfrac{\gamma_3^2 + \gamma_3(\kappa+2)}
{\gamma_3(n+1) + \kappa+2}
\quad \text{and} \quad
\gamma_3 > 0.$
\end{theorem}

\begin{proof}
Denote
\[
\gamma_3 = -2 - \min \mathrm{ev}(A_H).
\]
By Proposition~\ref{prop:corona graph of G and H}, we have
\[
\mathrm{QEC}(G \odot H)
=
\max\{\psi_H^{-1}(\mathrm{QEC}(G)),\gamma_3\}.
\]
Hence it suffices to show that
\[
\psi_H^{-1}(\mathrm{QEC}(G)) \le \gamma_3 \;\; \text{ and }\;\; \gamma_2 \le \gamma_3
\]
under each of the conditions \textit{(i)}--\textit{(iii)}. Since $\psi_H^{-1}(x)$ is
monotonically increasing on $(-\infty,\infty)$ and
\[
\psi_H^{-1}\!\left(
\frac{\gamma_3^2 + \gamma_3(\kappa+2)}
{\gamma_3(n+1)+\kappa+2}
\right)
= \gamma_3,
\]
the desired inequalities follow directly from the assumptions on
$\mathrm{QEC}(G)$.

\medskip
\noindent
\textit{(i)} Suppose
\[
\mathrm{QEC}(G) \le
\frac{\gamma_3^2 + \gamma_3(\kappa+1)}
{\gamma_3(n+1)+\kappa+2},
\qquad \gamma_3 < 0.
\]
By assumption, $\gamma_2=-\infty$. Moreover, since $\psi_H^{-1}$ is increasing, we obtain
\[
\psi_H^{-1}(\mathrm{QEC}(G))
\le
\psi_H^{-1}\!\left(
\frac{\gamma_3^2 + \gamma_3(\kappa+2)}
{\gamma_3(n+1)+\kappa+2}
\right)
=
\gamma_3.
\]
Hence,
\[
\mathrm{QEC}(G \odot H)
=
\max\{\psi_H^{-1}(\mathrm{QEC}(G)),\gamma_2,\gamma_3\}
=
\gamma_3.
\]

\medskip
\noindent
\textit{(ii)} Suppose
\[
\mathrm{QEC}(G) \le 0,
\qquad \gamma_3 > 0.
\]
Since $\psi_H^{-1}$ is increasing and $\psi_H^{-1}(0)=0$, we obtain
\[
\psi_H^{-1}(\mathrm{QEC}(G)) \le 0 < \gamma_3.
\]
Moreover, we have also
\[\gamma_2=0<\gamma_3.\]
Hence,
\[
\mathrm{QEC}(G \odot H)
=
\max\{\psi_H^{-1}(\mathrm{QEC}(G)), \gamma_2,\gamma_3\}
=
\gamma_3.
\]

\medskip
\noindent
\textit{(iii)} Suppose
\[
0 < \mathrm{QEC}(G) \le
\frac{\gamma_3^2 + \gamma_3(\kappa+2)}
{\gamma_3(n+1)+\kappa+2},
\qquad \gamma_3 > 0.
\]
By assumption, $\gamma_2=-\infty$. Moreover, since $\psi_H^{-1}$ is increasing, we obtain
\[
\psi_H^{-1}(\mathrm{QEC}(G))
\le
\psi_H^{-1}\!\left(
\frac{\gamma_3^2 + \gamma_3(\kappa+2)}
{\gamma_3(n+1)+\kappa+2}
\right)
=
\gamma_3.
\]
Hence,
\[
\mathrm{QEC}(G \odot H)
=
\max\{\psi_H^{-1}(\mathrm{QEC}(G)), \gamma_2,\gamma_3\}
=
\gamma_3.
\]
Therefore,
\[
\mathrm{QEC}(G \odot H)
=
-2 - \min \mathrm{ev}(A_H).
\]
This completes the proof.
\end{proof}

The following corollary is an immediate consequence of the previous theorems and gives a necessary and sufficient condition for the corona graph $G\odot H$ to belong to the QE class.

\begin{corollary}
\label{corollary IV : corona QE class}
Let $G=(V_1,E_1)$ be a connected graph with $|V_1|\ge 2$, and let 
$H=(V_2,E_2)$ be a $\kappa$-regular graph on $n=|V_2|$ vertices, where $n\geq 1$ and $\kappa\geq 0$. 
Let $A_H$ be the adjacency matrix of $H$. Then the corona graph $G \odot H$ belongs to the QE class if and only if $G$ belongs to the QE class and $\min \mathrm{ev}(A_H) \ge -2$.
\end{corollary}

\begin{remark}\normalfont
The results obtained in this section, together with those in 
\cite{Ferdi-Baskoro-Obata-Santika2025}, complete the determination of 
$\mathrm{QEC}(G\odot H)$ for all regular graphs $H$.
\end{remark}


\section{Quadratic Embedding Constants and the Second Distance Eigenvalue}

In \cite{Choudhury-Nandi2023}, it was shown that for the double star graph 
$K_2\odot \bar{K}_n$, which can be viewed as a special case of $G\odot H$ 
with $H$ being a regular graph, the quadratic embedding constant coincides 
with the second largest eigenvalue of its distance matrix. Motivated by this result, 
in this final section we characterize graphs $G$ for which the quadratic embedding 
constant of $G\odot H$ also coincides with the second largest eigenvalue of the 
distance matrix.

Moreover, it is known from \cite{Indulal-Stevanovic2015, Lin-Shu-Xue-Zhang2021} that 
the eigenvalues of $G\odot H$ can be described when $G$ is a distance-regular graph 
and $H$ is a regular graph. This class of graphs provides examples of graphs $G$ 
for which the above property holds.  

\begin{theorem}
Let $G=(V_1,E_1)$ be a connected graph with $|V_1|\ge 2$, and let 
$H=(V_2,E_2)$ be a $\kappa$-regular graph on $n=|V_2|$ vertices, where $n\geq 1$ and $\kappa\geq 0$. Then
\[
\mathrm{QEC}(G\odot H)=\delta_2(D[G\odot H])
\]
if and only if
\[
\mathrm{QEC}(G)=\delta_2(D_G).
\]
\end{theorem}

\begin{proof}
For simplicity, we write $D:=D[G\odot H]$.

\medskip
\noindent
($\Leftarrow$)
Assume that $\mathrm{QEC}(G)=\delta_2(D_G)$.
To prove the assertion, it suffices to determine the eigenvalues of $D$
arising from the spectrum of $H$ and from the interaction between $G$
and $H$. The following claims describe these eigenvalues.

\medskip
\noindent
\textbf{Claim 1.}
Let $H$ be a $\kappa$-regular graph with adjacency eigenvalues
\[
\kappa=\lambda_1 \ge \lambda_2 \ge \cdots \ge \lambda_n .
\]
Then $-2-\lambda_j$ is an eigenvalue of $D[G\odot H]$ for $j=2,\dots,n$.

Since $H$ is $\kappa$-regular, we have
\[
A_H\mathbf{1}=\kappa\mathbf{1},
\]
and the remaining eigenvectors of $A_H$ are orthogonal to $\mathbf{1}$.
Let $x\in\mathbb{R}^n$ be an eigenvector of $A_H$ corresponding to the eigenvalue $\lambda_j$ with $j\ge2$. Then
\[
A_Hx=\lambda_j x,
\qquad
\langle \mathbf{1},x\rangle=0.
\]
Define the vector $f$ by taking $\xi_o=0$ and
$\xi_i=x_i\mathbf{1}$ for all $i\in V_2$. Substituting these into the previously derived expression for $(D-\lambda)f$, we obtain
\begin{align*}
(D-\lambda)f
&=
\sum_{i\in V_2}D_G(x_i\mathbf{1})\otimes
\begin{bmatrix}1\\\mathbf{1}\end{bmatrix}
+
\sum_{i\in V_2}x_i\mathbf{1}\otimes
\begin{bmatrix}0\\-(A_H+2)e_i\end{bmatrix}
-
\sum_{i\in V_2}\lambda x_i\mathbf{1}\otimes
\begin{bmatrix}0\\e_i\end{bmatrix}.
\end{align*}
Since $\sum_{i\in V_2}x_i=0$, the first term vanishes. Hence
\[
(D-\lambda)f
=
\sum_{i\in V_2}x_i\mathbf{1}\otimes
\begin{bmatrix}0\\-(A_H+2+\lambda)e_i\end{bmatrix}.
\]
Thus $(D-\lambda)f=0$ holds whenever
\[
(A_H+2+\lambda)x=0.
\]
Since $A_Hx=\lambda_jx$, this condition is equivalent to
\[
\lambda=-2-\lambda_j.
\]
Therefore $f$ is an eigenvector of $D[G\odot H]$ with eigenvalue $-2-\lambda_j$.

\medskip
\noindent
\textbf{Claim 2.}
Every root of
\begin{align}\label{eq root of lambda=psi}
\lambda^2-((n+1)\delta_2(D_G)-\kappa-2)\lambda-(\kappa+2)\delta_2(D_G)=0
\end{align}
is an eigenvalue of $D[G\odot H]$.

Let $x$ be a unit eigenvector of $D_G$ corresponding to the eigenvalue
$\delta_2(D_G)$ satisfying
\[
D_Gx=\delta_2(D_G)x,
\qquad
\langle \mathbf1,x\rangle=0,
\qquad
\langle x,x\rangle=1,
\]
which exists since $\mathrm{QEC}(G)=\delta_2(D_G)$. Define
\[
\xi_o=\alpha x, \qquad \xi_i=x, \qquad i\in V_2,
\]
where
\[
\alpha=\frac{n\delta_2(D_G)}{\lambda-\delta_2(D_G)}.
\]
Substituting these into the expression for $(D-\lambda)f$, we obtain
\begin{align*}
(D-\lambda)f
&=\alpha D_Gx\otimes 
\begin{bmatrix}1\\\mathbf1\end{bmatrix}
+nD_Gx\otimes 
\begin{bmatrix}1\\\mathbf1\end{bmatrix}
+x\otimes 
\begin{bmatrix}0\\-(A_H+2)\mathbf1\end{bmatrix}
-\lambda\alpha x\otimes 
\begin{bmatrix}1\\0\end{bmatrix}
-\lambda x\otimes 
\begin{bmatrix}0\\\mathbf1\end{bmatrix}.
\end{align*}
Since $D_Gx=\delta_2(D_G)x$ and $H$ is $\kappa$-regular, we have
\[(A_H+2)\mathbf1=(\kappa+2)\mathbf1.\]
Hence,
\begin{align*}
(D-\lambda)f
&=\delta_2(D_G)(\alpha+n)x\otimes
\begin{bmatrix}1\\\mathbf1\end{bmatrix}
-(\kappa+2)x\otimes
\begin{bmatrix}0\\\mathbf1\end{bmatrix}-\lambda\alpha x\otimes
\begin{bmatrix}1\\0\end{bmatrix}
-\lambda x\otimes
\begin{bmatrix}0\\\mathbf1\end{bmatrix}.
\end{align*}
A direct computation shows that $(D-\lambda)f=0$ holds precisely when
$\lambda$ satisfies the quadratic equation \eqref{eq root of lambda=psi}.
Therefore every root of \eqref{eq root of lambda=psi} is an eigenvalue of
$D[G\odot H]$.

By Proposition~\ref{prop:corona graph of G and H}, we have
\[
\mathrm{QEC}(G\odot H)
=
\max\{\psi_H^{-1}(\mathrm{QEC}(G)),\gamma_2,\gamma_3\}.
\]
Since $\mathrm{QEC}(G)=\delta_2(D_G)$, it follows that
\[
\mathrm{QEC}(G\odot H)
=
\max\{\psi_H^{-1}(\delta_2(D_G)),\gamma_2,\gamma_3\}.
\]
By \textbf{Claim~1}, the values $-2-\lambda_j$ $(j=2,\dots,n)$ are eigenvalues of $D[G\odot H]$, and hence $\gamma_2$ and $\gamma_3$ are an eigenvalue of $D[G\odot H]$. 
Moreover, by \textbf{Claim~2}, $\psi_H^{-1}(\delta_2(D_G))$ is also an eigenvalue of $D[G\odot H]$. 
Therefore
\[
\max\{\psi_H^{-1}(\delta_2(D_G)),\gamma_2,\gamma_3\}
\]
is an eigenvalue of $D[G\odot H]$. Since 
\[
\delta_2(D[G\odot H])
\le
\mathrm{QEC}(G\odot H)
<
\delta_1(D[G\odot H]).
\]
Consequently,
\[
\mathrm{QEC}(G\odot H)=\delta_2(D[G\odot H]).
\]

\medskip
\noindent
($\Rightarrow$)
Assume that
\[
\mathrm{QEC}(G\odot H)=\delta_2(D[G\odot H]).
\]
Then there exists a vector
\[f=\xi_o\otimes \begin{bmatrix} 1 \\ 0 \end{bmatrix}
+ \sum_{i\in V_2} \xi_i\otimes \begin{bmatrix} 0 \\ e_i \end{bmatrix},
\qquad
\xi_i \in C(V_1),
\quad i\in \{o\}\cup V_2\,.\] satisfying
\[
(D-\lambda)f=0, \qquad 
\langle \mathbf1,f\rangle=0, \qquad 
\langle f,f\rangle=1,
\]
where $\lambda=\mathrm{QEC}(G\odot H)$.
From the block structure of $D$, this is equivalent to the
system
\begin{align}
&(D_G-J_G-\lambda)\xi_o 
+\sum_{i\in V_2}D_G\xi_i=0,
\label{05eqn:basic equation (11)}
\\
&\lambda\xi_o
=\sum_{j\in V_2} \langle e_i, (A_H+2+\lambda)e_j\rangle \xi_j,
\qquad i\in V_2,
\label{05eqn:basic equation (12)}
\\
&\langle \mathbf{1}, \xi_o\rangle
+\sum_{i\in V_2} \langle \mathbf{1}, \xi_i\rangle=0,
\label{05eqn:basic equation (13)}
\\
&\langle \xi_o, \xi_o\rangle
+\sum_{i\in V_2} \langle \xi_i, \xi_i\rangle=1.
\label{05eqn:basic equation (14)}
\end{align}

First suppose that $A_H+2+\lambda$ is singular.
Then $\lambda=-2-\theta$ for some adjacency eigenvalue
$\theta\in \mathrm{ev}(A_H)$.
Hence $\lambda\in\{\gamma_2,\gamma_3\}$. Next assume that $A_H+2+\lambda$ is nonsingular.
From \eqref{05eqn:basic equation (12)} we obtain
\[
\xi_i
=\lambda
\langle e_i,(A_H+2+\lambda)^{-1}\mathbf1\rangle
\xi_o,
\qquad i\in V_2.
\]
Substituting this into \eqref{05eqn:basic equation (13)} yields
\[
\bigl(
1+\lambda\langle \mathbf1,(A_H+2+\lambda)^{-1}\mathbf1\rangle
\bigr)
J_G\xi_o=0.
\]
If
\[
1+\lambda\langle \mathbf1,(A_H+2+\lambda)^{-1}\mathbf1\rangle=0,
\]
then \eqref{05eqn:basic equation (11)} implies
\[
(J_G+\lambda)\xi_o=0.
\]
Since $\lambda>-1$ and $\lambda\ne 0$, the matrix $J_G+\lambda$ is nonsingular,
and hence $\xi_o=0$.
Consequently $\xi_i=0$ for all $i\in V_2$, which contradicts
\eqref{05eqn:basic equation (14)}.
Therefore
\[
1+\lambda\langle \mathbf1,(A_H+2+\lambda)^{-1}\mathbf1\rangle\neq0,
\]
and hence $J_G\xi_o=0$.
Using \eqref{05eqn:basic equation (11)} and \eqref{05eqn:basic equation (12)},
we obtain
\[
D_G\xi'-\psi_H(\lambda)\xi'=0,
\qquad
\langle \mathbf1,\xi'\rangle=0,
\qquad
\langle \xi',\xi'\rangle=1,
\]
where
\[
\xi'
=
\left(
1+\lambda^2
\|(A_H+2+\lambda)^{-1}\mathbf1\|^2
\right)^{-1/2}
\xi_o
\]
and
\[
\psi_H(\lambda)
=
\frac{\lambda}{
1+\lambda\langle \mathbf1,(A_H+2+\lambda)^{-1}\mathbf1\rangle}.
\]
Thus, we have $\psi_H(\lambda)$ is an eigenvalue of $D_G$ corresponding to a vector orthogonal to
$\mathbf1$.
Hence
\[
\psi_H(\lambda)\le \mathrm{QEC}(G).
\]
Since $\lambda=\mathrm{QEC}(G\odot H)$, we obtain
\[
\psi_H(\lambda)=\mathrm{QEC}(G).
\]
Finally, since
\[
\delta_2(D_G)\le \mathrm{QEC}(G)<\delta_1(D_G),
\]
we conclude that
\[
\mathrm{QEC}(G)=\delta_2(D_G).
\]
\end{proof}

\begin{remark}
Suppose that $G$ is a connected graph on two or more vertices and that
$H$ is a $\kappa$-regular graph. Then the adjacency matrix $A_H$
satisfies $A_H\mathbf{1}=\kappa\mathbf{1}$, and the remaining
adjacency eigenvectors are orthogonal to $\mathbf{1}$. 

Under the assumptions of Theorem~\ref{theorem corona graph with lambda equal zero}
and Theorem~\ref{theo III : QEC of corona graph G and H - min eigenvalue case},
the quadratic embedding constant of the corona graph $G\odot H$
coincides with the second largest eigenvalue of the distance matrix
$D[G\odot H]$. That is,
\[
\mathrm{QEC}(G\odot H)=\delta_2(D[G\odot H]).
\]
Indeed, in Theorem~\ref{theorem corona graph with lambda equal zero}
we obtain \[\mathrm{QEC}(G\odot H)=0,\] whereas in
Theorem~\ref{theo III : QEC of corona graph G and H - min eigenvalue case}
we have
\[
\mathrm{QEC}(G\odot H)=-2-\min \mathrm{ev}(A_H).
\]
Since $H$ is regular, these values correspond to distance eigenvalues
arising from the spectrum of $A_H$. Under the stated conditions they
lie strictly between the largest and the remaining eigenvalues of
$D[G\odot H]$, and therefore they coincide with the second largest
eigenvalue of the distance matrix.
\end{remark}


\section*{Declarations}

\textbf{Funding} This research was funded by the Indonesian Endowment Fund for Education (LPDP)
on behalf of the Indonesian Ministry of Higher Education, Science and Technology,
and is managed under the EQUITY Program (Contract No. 4298/B3/DT.03.08/2025), and the PMDSU Program funded by the Ministry of Higher Education, Science, and Technology, Indonesia.

\medskip
\noindent
\textbf{Conflict of interest} The authors declare that they have no conflict of interest.

\medskip
\noindent
\textbf{Data Availability} The manuscript does not contain any associated data.





\bibliography{sn-bibliography}

\end{document}